\documentclass[letterpaper, 10 pt, conference]{ieeeconf}  

\IEEEoverridecommandlockouts  

\usepackage{cite}
\usepackage{amsmath,amssymb,amsfonts}

\usepackage{amsthm}

\newtheorem{theorem}{Theorem}
\newtheorem{lemma}{Lemma}

\newtheorem*{remark*}{Remark}

\newtheorem{definition}{Definition}
\newtheorem{prop}{Proposition}
\newtheorem{example}{Example}
\newtheorem*{problem*}{Problem}
\usepackage[linesnumbered]{algorithm2e}
\usepackage{graphicx}
\usepackage{textcomp}
\usepackage{xcolor}
\usepackage{cleveref}

\title{\LARGE \bf
Uncertainty Intervals for Robust Bottleneck Assignment}

\author{Elad Michael$^{*}$, Tony A. Wood, Chris Manzie, and Iman Shames
\thanks{$^{*}$All authors are with Department of Electrical and Electronic Engineering at the University of Melbourne
{\tt\small eladm@student.unimelb.edu.au,
        \{wood.t,manziec,iman.shames\}@unimelb.edu.au}}%
}

\begin{document}

\maketitle
\thispagestyle{empty}
\pagestyle{empty}

\begin{abstract}

We examine the robustness of bottleneck assignment problems to perturbations in the assignment weights. We derive two algorithms that provide uncertainty bounds for robust assignment. We prove that the bottleneck assignment is guaranteed to be invariant to perturbations which lie within the provided bounds. We apply the method to an example of task assignment for a multi-agent system.
\end{abstract}

\section{INTRODUCTION}

%

This paper focuses on quantifying the set of weight perturbations which maintain the optimality of the original assignment. Task assignment is a central part of autonomous multi-agent systems, assigning agents to tasks in such a way as to optimize some measure of the quality of the assignment. We model the problem with a graph, with the agents and tasks as vertices and the assignments as weighted edges. The optimization cost can be any of minimum weight, minimum total weight, balanced weight, or one of many other types of assignment problems. In many scenarios, especially in physical applications, the weight of assigning an agent to a task has some associated uncertainty. This could be due to a physical measurement error such as a radar covariance, a dynamic environment within which the decisions are being made, or simply decimal roundoff errors. The optimal assignment is robust to some set of perturbations, and this paper provides a parameterization of that set of allowable perturbations for the bottleneck assignment problem.\\

This problem is studied under various names, such as ``sensitivity"\cite{lin2007sensitivity}, ``stability"\cite{klaus2013paths}, or ``post-optimality"\cite{KINDERVATER198576}. We use \emph{robustness} to refer to the range of variations in the problem parameters under which we can certify the optimal solution. An overview of modern linear sensitivity analysis and its application is covered in \cite{jansen1997sensitivity}. In \cite{jansen1997sensitivity}, they differentiate between several types of sensitivity analysis, focused on different aspects of the optimal solution such as optimal cost, active constraints, and the slope of the optimal cost function. The assignment problem, due to its additional constraint structure, is best approached using what is defined as ``Type 2 Sensitivity Analysis". However, \cite{jansen1997sensitivity} focuses on perturbations of a single variable. We will use the term ``robustness" to describe the effects of uncertainty in weight on all of the assignments.\\

In this paper we focus on the robustness of the bottleneck assignment problem (BAP). In a matching of agents to tasks, each with it's associated weight, the BAP seeks to minimize the maximum weight within a matching. In \cite{8264324} the BAP to assign multiple decoys to multiple incoming threats. The bottleneck property minimizes the worst case time it takes each decoy to seduce and divert the threat away from the asset. This ensures that all threats are dealt with in the minimum time. The BAP is used in other minimum time applications,such as \cite{Ravindran1977} where they give the example of transporting perishable goods or military supplies reaching command posts during an emergency.\\

An algorithm with robustness guarantees is provided in \cite{volgenant2010improved} along with several complexity improvements on others' work. However, this approach is used to find robust solutions, not to quantify the robustness of the optimal solution to the original problem. Robustness for the linear assignment problem to correlated perturbations in adjacent to a single node is examined in \cite{lin2007sensitivity}, reflecting an error in the state of an agent or a task. In \cite{sotskovStability1995}, robustness/stability/sensitivity of the solution for the case where all perturbations are bounded by the same bound is established. This differs from the results of this paper where perturbation bounds are not uniform for all weights. Additionally, contrasting with \cite{sotskovStability1995}, the algorithms for computing these bounds are presented. \\

Our contribution in this paper is finding individual perturbation bounds for each weight within which we certify the optimality of the original assignment solution. We allow all weights to vary independently and simultaneously, rather than restricting variations to a subset of the edges. In the presence of uncertainty, these bounds may provide sufficient robustness to certify that the assignment made was optimal. However, it may be that the assignments made are not robust enough for the user's needs. In this case, the algorithms proposed here can only alert the user to that fact, they do not propose a more robust assignment solution. \\




The paper is organized as follows. Section 2 is devoted to background on the assignment problem and a formal definition of robustness. Section 3 will provide the lower bounds for the bottleneck assignment problem, and the algorithms to find them. Section 4 covers numerical examples and implementation. Section 5 concludes and remarks on further analysis.\\

\section{Problem Formulation}

Define a graph $\mathcal{G}=(\mathcal{V},\mathcal{E},\mathcal{W})$ over which we will solve the assignment problem. Let the set of vertices be $\mathcal{V} = \mathcal{V}_A \cup \mathcal{V}_T$ such that $\mathcal{V}_A \cap \mathcal{V}_T = \varnothing$ and $|\mathcal{V}_A| \leq |\mathcal{V}_T|$, the edge set $\mathcal{E} \subseteq \mathcal{V}_A \times \mathcal{V}_T$, and edge weights $\mathcal{W} = \{ w_{ij} \in \mathbb{R} \mid (i,j) \in \mathcal{E} \}$. We also define a set of binary decision variables variable $\{\pi _{ij} \mid \pi_{ij} \in \{0,1\},(i,j) \in \mathcal{E}\}$ to indicate if the edge from vertex $i$ to $j$ was assigned. With this graph and these decision variables, the bottleneck assignment problem can be formulated as

\begin{subequations}
\begin{align}
\underset{\pi}{\min}\phantom{ba} \underset{(i,j) \in \mathcal{E}}{\max}\phantom{ba}& \pi _{ij} w_{ij} \\
\textrm{subject to } &\sum_{i \in \mathcal{V}_A} \pi_{ij} = 1,\phantom{ba} j \in \mathcal{V}_T \label{eq:fullMatching}\\
&\sum_{j \in \mathcal{V}_T} \pi_{ij} \leq 1,\phantom{ba} i \in \mathcal{V}_A. \label{eq:taskDominant}
\end{align}
\end{subequations}

If $\pi_{ij} = 1$ then we say that vertex $i$ is assigned to vertex $j$. In the rest of the paper the double indices associated with edge $(i,j) \in \mathcal{E}$ will be represented as a single index. Constraints \eqref{eq:fullMatching} and \eqref{eq:taskDominant} ensure that every vertex is assigned to at most one other vertex. Two edges are called adjacent if they share a common vertex. A matching in $\mathcal{G}$ is a set of pairwise non-adjacent edges. That is, no two edges share a common vertex. A \emph{maximum matching} $\mathcal{M}$ in $\mathcal{G}$ is a matching with the largest number of edges. The bottleneck problem finds the maximum matching on the graph which minimizes the worst case edge weight. In this paper we focus on maximal cardinality bipartite matchings, where in at least one of the vertex sets $\mathcal{V}_A$ all vertices are included in the matching. However, for brevity, we will refer to maximal cardinality bipartite matchings as maximal matchings.\\

The bottleneck assignment problem is well studied\cite{burkard1999linear}, and there are a many algorithms that solve it. The threshold algorithm, Algorithm~\ref{alg:bottleneck} below, solves the bottleneck assignment problem\cite{burkard1999linear}, and provides intuition used later in the paper. The algorithm first sorts the edge set in descending according to their weights (line 1). If a maximum matching exists within the sorted edge set (line 3), then the edge with the greatest weight is removed (line 5). Checking for a maximum matching and removing the greatest weight edge is repeated until no maximum matching exists. The algorithm terminates with the last removed edge being the bottleneck edge. 
\begin{algorithm}
\caption{Threshold Algorithm for solving the BAP \label{alg:bottleneck}}
 \KwData{$\mathcal{G}=(\mathcal{V},\mathcal{E},\mathcal{W})$}
 \KwResult{bottleneckAssignment}
 $\bar{\mathcal{E}} \leftarrow \mathcal{E}$\;
 \While{Maximum Matching $\in$ $\bar{\mathcal{E}}$}{
  $e \leftarrow \underset{\bar{e}\in \bar{\mathcal{E}}}{\text{argmax}} \phantom{bo} \mathcal{W}$ \;
  $\bar{\mathcal{E}} \leftarrow \bar{\mathcal{E}}\setminus \{e\}$\;
 }
 return e $\backslash \backslash$ \emph{The bottleneck edge} 
\end{algorithm}
Algorithm~\ref{alg:bottleneck} has complexity $\mathcal{O}(T(N)N)$ where $T(N)$ is the complexity for checking the existence of a maximum matching in the set of edges $|\mathcal{E}| = N$. The most common method is the Hopcroft-Karp Algorithm, which has complexity $\mathcal{O}(|\mathcal{E}|\sqrt{|\mathcal{V}|})$ \cite{burkard1999linear}. There are faster probabilistic methods, see \cite{burkard1999linear} for a survey of methods. This gives Algorithm~\ref{alg:bottleneck} complexity $\mathcal{O}(N^2\sqrt{|\mathcal{V}|})$. Note that Algorithm~\ref{alg:bottleneck} is given here only for intuition and context, the methods described below pertain only to the robustness analysis of the optimal solution. The actual algorithm used to find the bottleneck is irrelevant to our results.\\

To formally define robustness related to the bottleneck assignment problem, we first define the perturbed graph. 
\begin{definition}
Given a graph $\mathcal{G}=(\mathcal{V},\mathcal{E},\mathcal{W})$, define the perturbed graph $\bar{\mathcal{G}}=(\mathcal{V},\mathcal{E},\bar{\mathcal{W})}$ such that  $\bar{\mathcal{W}} = \{ w_e + \delta_e \mid e\in\mathcal{E} \}$, where $w_e \in \mathcal{W}$ is the weight of edge $e$ from the original graph and $\delta_e\in D$ is some perturbation from a set of scalar perturbations $D \subseteq \mathbb{R}$. \label{def:perturbedGraph}
\end{definition}

Define the mapping $A(\mathcal{G})$ which takes a graph and returns the index of the bottleneck edge. If there are multiple equivalent bottleneck edges, then the mapping will the return the set of them.
\begin{definition}
Let $e^*\in A(\mathcal{G})$ be the bottleneck edge of the graph $\mathcal{G}$. If $e^* \in A(\bar{\mathcal{G}})$, for the perturbed weight graph in Definition~\ref{def:perturbedGraph}, then the bottleneck assignment $e^*$ is robust to the set of perturbations $\{ \delta_e \mid e \in \mathcal{E}\}$. \label{def:Robustness}
\end{definition}
With this definition of robustness, the problem can be stated as follows.
\begin{problem*}
Find the set of allowable perturbations $\Lambda$ such that the bottleneck assignment $e^* \in \mathcal{G}$ is robust to the perturbation $\lambda = \{ \delta_e \mid e \in \mathcal{E}\}$ if and only if $\lambda \in \Lambda$.
\end{problem*}

\begin{remark*}
The parametrization of the set $\Lambda$ as defined in the problem statement has tangible repercussions for task assignment under uncertainty. For example, suppose the graph represents a task assignment in a physical multi-agent system. The agents and tasks are the vertices, and the edges are the possible assignments of agents to tasks, each with an associated weight. If the uncertainties associated with the assignment weights fall within the set of allowable perturbations $\Lambda$, then the bottleneck assignment is guaranteed to remain optimal, preventing any need for reassignment.
\end{remark*}

\section{Robustness of the Bottleneck}

To parametrize the set of allowable perturbations $\Lambda$, we first propose an optimization problem to find a subset $\Lambda_\Delta = \{\delta_e \mid |\delta_e| \leq \Delta , e \in \mathcal{E}\} \subseteq \Lambda$
\begin{subequations}
\begin{align}
\underset{\Delta \in \mathbb{R}}{\max} \phantom{bl} &\Delta \label{eq:deltaMax}\\ 
\text{s.t.} \phantom{bla} &A(\bar{\mathcal{G}}) = A(\mathcal{G}),  \label{eq:2b}\\ 
&\bar{\mathcal{W}} = \{w_e + \delta_e \mid w_e \in \mathcal{W}, |\delta_e| \leq \Delta, e \in \mathcal{E}\},
\end{align}
\end{subequations}
where we have continued to use the graph notation presented in Definition~\ref{def:perturbedGraph}.\\

To find $\Lambda _\Delta$, we first define three bottleneck edges
\begin{align}
&e^- \in A(\mathcal{G}^-), \label{eq:i-},\\
&e^* \in A(\mathcal{G}), \label{eq:i},\\
&e^+ \in A(\mathcal{G}^+) \label{eq:i+}.
\end{align}
These are the bottleneck edges of the original graph $\mathcal{G}$ and of two subgraphs $\mathcal{G}^-$ and $\mathcal{G}^+$. Define $\mathcal{G}^-$ to be the graph constructed by removing the vertices, edges, and weights adjacent to and including $e^*$ and $\mathcal{G}^+$ to be constructed by removing only the edge and weight $e^*,w_{e^*}$ from the original graph. The differences between the weights of these three bottleneck edges $e^-$,$e^*$, and $e^+$ will bound the allowable perturbations, leading to the parameterization of $\Lambda_\Delta$.
\begin{lemma}
Consider the indices $e^-$,$e^*$, and $e^+$, defined in \eqref{eq:i-} to \eqref{eq:i+}. Then $w_{e^-} \leq w_{e^*}$, and if $e^+$ exists then $w_{e^*} \leq w_{e^+}$.\\
\label{lm:Ordering}
\end{lemma}
\begin{proof}
First we will prove that if $e^+$ exists, $w_{e^*} \leq w_{e^+}$. By the definition of $e^*$, there is no maximum matching on the original graph $\mathcal{G}$ composed of edges that all have weight less than $w_{e^*}$. Since $w_{e^+}$ is an edge with the greatest weight in a maximum matching on the graph $\mathcal{G^+}$, which is missing the edge $e^*$, $w_{e^+}$ must be greater than or equal to $w_{e^*}$.\\

Next, to obtain a contradiction, assume $w_{e^-}$ is greater than $w_{e^*}$. By the definition of $e^*$, there exists a maximum matching on $\mathcal{G}$ with $e^*$ the greatest weight edge in that matching. Let $\pi^* \subseteq \mathcal{E}$ represent the edges in that matching. The set $\pi ^* \setminus \{e^*\} \subseteq \mathcal{E}^-$ and is a maximum matching on $\mathcal{G}^-$. However, the edge with the greatest weight in $\pi ^* \setminus \{e^*\}$ has weight less than or equal to $w_{e^*}$, which contradicts $w_{e^-}$ being greater than $w_{e^*}$.
\end{proof}
\bigbreak


Using the concept of the subgraphs $\mathcal{G}^-$ and $\mathcal{G}^+$, we can bound $\Delta$ from \eqref{eq:deltaMax}, and construct $\Lambda _\Delta$.
\begin{theorem}
Let $w_{e^*}$ be the weight of the bottleneck edge $e^*$ on the graph $\mathcal{G}$, and let $w_{e^-}$, and $w_{e^+}$ be the bottleneck weights of the subgraphs $\mathcal{G}^-$ and $\mathcal{G}^+$, respectively, as defined in \eqref{eq:i-} to \eqref{eq:i+}. The bottleneck assignment is robust, in the sense of Definition~\ref{def:Robustness}, to perturbations in $\Lambda _\Delta$ where
\begin{equation}
0 \leq \Delta \leq \frac{1}{2}\min(w_{e^*}-w_{e^-},w_{e^+}-w_{e^*}).
\label{eq:DeltaRobust}
\end{equation}

In the case that the graph $\mathcal{G}^+$ as defined in \eqref{eq:i+} does not have a maximal matching, the perturbations are bounded by

\begin{equation}
0 \leq \Delta \leq \frac{1}{2}(w_{e^*}-w_{e^-}). 
\end{equation}
\label{thm:adjBottlenecks}
\end{theorem}

\begin{proof}
We will prove by contradiction that no edge $x$ with weight $w_x \neq w_{e^*}$ will be the bottleneck of the perturbed weight graph under some set of allowable perturbations.\\

Let $\bar{w}_i= w_e + \delta_e$ be weight of edge $e$ after perturbation. Note that $e^*$ and $e^-$ guarantee the existence of a maximum matching with greatest weight $w_{e^*}$ and second greatest weight $w_{e^-}$. Let $\pi^*$ be the set of edges corresponding to that maximum matching.\\

To obtain a contradiction, let the edge $x \neq e^*$ be the bottleneck of the perturbed weight graph $\bar{\mathcal{G}}$ as in Definition~\ref{def:perturbedGraph}, with perturbed weight $\bar{w}_x \neq \bar{w}_{e^*}$.\\

We will first establish that $\bar{w}_x$ must be strictly less than $\bar{w}_{e^*}$. By the definition of the bottleneck, $\bar{w}_x$ is less than or equal to the greatest weight edge in every maximum matching on the perturbed weight graph. Therefore, $\bar{w}_x$ is less than or equal to the greatest weight edge in the set $\pi ^*$. By Lemma~\ref{lm:Ordering} and the bound on the perturbations, $\bar{w}_{e^*}$ is the greatest weight edge in the matching $\pi ^*$. Therefore, $\bar{w}_x < \bar{w}_{e^*}$. Additionally, because $\bar{w}_x$ is the bottleneck edge, it must be the greatest weight edge in a maximum matching $\pi ^x$, and thus each edge in $\pi ^x$ must have a perturbed weight less than $\bar{w}_{e^*}$ and clearly $e^* \notin \pi^x$. \\

We will now establish that since $\bar{w}_x < \bar{w}_{e^*}$, it cannot be the bottleneck edge. Because of the bound on the perturbations, the only edges which can be perturbed to have weight less than $\bar{w}_{e^*}$ are those edges which have unperturbed weight $w_j < w_{e^+}$. Therefore, every edge in the maximal matching $\pi ^x$ must have an unperturbed weight less than $w_{e^+}$. However, by the definition of $e^+$, there is no maximal matching composed of weights less than $w_{e^+}$ that does not contain $e^*$. Therefore, $\pi ^x$ does not exist and $\bar{w}_x$ cannot be the new bottleneck edge.

\end{proof}
\bigbreak

\begin{example} 
Consider a scenario where $|\mathcal{V_A}|=|\mathcal{V_T}|=3$ and $\mathcal{E}=\mathcal{V}_A\times \mathcal{V}_T$. This could represent a physical scenario with a fully connected set of three agents and tasks. The weight set $\mathcal{W}$ is represented by a matrix $C$, where $w_{ij}$ is its $ij$-th element. This is a convenient representation and will be used for the rest of the paper. This matrix is termed the cost matrix and for this example is \\
\begin{align*} C = 
\begin{bmatrix}
 3 & 2 & 1 \\
 4 & 5 & 6 \\
 9 & 8 & 7
\end{bmatrix}.
\end{align*}
Algorithm~\ref{alg:bottleneck} can be applied to this graph. The sorted list of weights is clearly the integers from $9$ to $1$, and initially there are six maximum matchings in this graph. Removing the greatest weight edge leaves 4 maximum matchings in the graph. Removing the greatest weight edge again leaves us with the graph whose corresponding cost matrix is
\begin{align*} C = 
\begin{bmatrix}
 3 & 2 & 1 \\
 4 & 5 & 6 \\
 \infty & \infty & 7
\end{bmatrix}.
\end{align*}
where $\infty$ represents the absence of an edge.\\

Here if we were to remove the greatest weight edge there would be no maximum matchings available, so the edge with weight $7$ is the bottleneck. We can now define the cost matrices of the two subgraphs $\mathcal{G}^-$ and $\mathcal{G}^+$ to be
\begin{align*}
C^- = 
\begin{bmatrix}
 3 & 2 \\
 4 & 5 \\
\end{bmatrix}, 
C^+ = 
\begin{bmatrix}
 3 & 2 & 1 \\
 4 & 5 & 6 \\
 9 & 8 & \infty 
\end{bmatrix}.
\end{align*}
Running Algorithm~\ref{alg:bottleneck} on these two subgraphs, we find $e^-$ and $e^+$ to be the edges with weight 4 and 8 respectively. In accordance with Lemma~\ref{lm:Ordering}, $w_{e^-}=4 \leq w_{e^*}=7 \leq w_{e^+}=8$. Using \eqref{eq:DeltaRobust}, we find the $\Delta$ value for this graph to be $\frac{1}{2}$. Therefore, the graph $\mathcal{G}$ is robust to any set of weight perturbations such that no perturbation has magnitude greater than or equal to $\frac{1}{2}$.\hfill $\square$ \\
\end{example}

Equation~\eqref{eq:DeltaRobust} states that as long the perturbations are all within the interval $[-\Delta,\Delta]$, the optimal assignment remains unchanged. This results coincides with the result presented in \cite{sotskovStability1995}. However, by relaxing these constraints, allowing the intervals to be unique and asymmetric, we are able to expand this set of allowable perturbation intervals.\\

Note that in the proof for Theorem~\ref{thm:adjBottlenecks}, the salient property of $\Delta$ is that edges with weight $w_j \geq w_{e^+}$ cannot be perturbed to have weight $\bar{w}_j < \bar{w}_{e^*}$, and edges with weight $w_j \leq w_{e^-}$ cannot be perturbed to have weight $\bar{w}_j > \bar{w}_{e^*}$. However, this property can be maintained with larger and asymmetric bounds on the perturbations. Additionally, the edges which are between $e^-,e^*,e^+$ are irrelevant to the method, they do not change the bottleneck edge regardless of their weight. Therefore, for edges between $e^-$ and $e^+$ other than $e^*$, we need not bound their perturbations at all. This leads us to Algorithm~\ref{alg:slowBound} below.

\begin{algorithm}
 \KwData{($e^*,\mathcal{G})$}

 $e^- \leftarrow A(\mathcal{G}^-)$ \;
 $e^+ \leftarrow A(\mathcal{G}^+)$ \;

 $\mathcal{U}$  $\backslash \backslash$ \emph{Upper bound on the allowable perturbations} \;
 $\mathcal{L}$  $\backslash \backslash$ \emph{Lower bound on the allowable perturbations} \;
 $\Delta^+ \leftarrow \frac{w_{e^+}-w_{e^*}}{2}$ \;
 $\Delta^- \leftarrow \frac{w_{e^*}-w_{e^-}}{2}$ \;
\For{$w_j\in\mathcal{W}$}{

	\uIf{$w_j \geq w_{e^+}$}{
		$\mathcal{L}[j] = (w_{e^*} + \Delta^+) - w_j$ \;
	}
	\uElseIf{$w_j \leq w_{e^-}$}{
		$\mathcal{U}[j] = (w_{e^*} - \Delta^-) - w_j$ \;
	}
	\uElseIf{$j \neq e^*$}{
		$\mathcal{L}[j] = -\infty$ \;
		$\mathcal{U}[j] = \infty$ \;
	}
	\Else{
		$\mathcal{L}[j] = -\Delta$ \;
		$\mathcal{U}[j] = \Delta$ \;
	}
}
return $\mathcal{U},\mathcal{L}$ \;

\caption{Relaxed $\Delta$ Algorithm for Robustness \label{alg:slowBound}}
\end{algorithm}

\begin{prop}
The bottleneck assignment is robust, in the sense of Definition~\ref{def:Robustness}, to the set of perturbations defined in Algorithm~\ref{alg:slowBound}.
\end{prop}

\begin{proof}
Let $e^-$,$e^*$, and $e^+$ be defined as in \eqref{eq:i-} to \eqref{eq:i+}, and let $\Delta^+$ and $\Delta^-$ be defined by Algorithm~\ref{alg:slowBound}. Let $\bar{w}_j$ represent the perturbed weight of edge $j$. Define the sets of upper and lower bounds returned by Algorithm~\ref{alg:fastBound} as $\mathcal{U}$ and $\mathcal{L}$. Let $(\mathcal{L}_j,\mathcal{U}_j)$ represent the allowable perturbation for edge $j$, such that $\bar{w}_j - w_j \in (\mathcal{L}_j,\mathcal{U}_j)$.\\

If the bounds provided by Algorithm~\ref{alg:slowBound} guarantee $\bar{w}_j \geq \bar{w}_{e^*}$ for any edge $j$ with weight $w_j \geq w_{e^+}$, and $\bar{w}_j \leq \bar{w}_{e^*}$ for any edge $j$ with weight $w_j \leq w_{e^-}$, then the proof for Theorem~\ref{thm:adjBottlenecks} applies to the set of bounds from Algorithm~\ref{alg:slowBound}.

For any edge $j$ with weight $w_j \geq w_{e^+}$,
\begin{align*}
\bar{w}_j > w_j+\mathcal{L}_j = w_j + (w_{e^*} + \Delta^+ - w_j) > \bar{w}_{e^*}. 
\end{align*}

For any edge $j$ with weight $w_j \leq w_{e^-}$,
\begin{align*}
\bar{w}_j < w_j+\mathcal{U}_j = w_j + (w_{e^*} - \Delta^- - w_j) < \bar{w}_{e^*}. 
\end{align*}
Therefore, the perturbation intervals provided by Algorithm~\ref{alg:slowBound} have the same properties as those provided in Theorem~\ref{thm:adjBottlenecks}, and the bottleneck assignment is robust to perturbations within the bounds provided by Algorithm~\ref{alg:slowBound}.

\end{proof}
Consider the case where an edge with weight greater than the bottleneck is subject to a positive perturbation, or a an edge with weight less than the bottleneck edge is subject to a negative perturbation. The bottleneck assignment is clearly robust to these types perturbations. This intuition provides additional insight into the set of allowable perturbations $\Lambda$.
\begin{prop}
Let $w_x$ be the weight of an edge $x\in\mathcal{E}$, $\bar{w}_x$ be the perturbed weight of that edge, and $e^* \in \mathcal{E}$ be the bottleneck as defined in \eqref{eq:i}. The bottleneck assignment is robust, in the sense of Definition~\ref{def:Robustness}, to any perturbations for which $\bar{w}_x > \bar{w}_{e^*}$ iff $w_x > w_{e^*}$, and $\bar{w}_x < \bar{w}_{e^*}$ iff $w_x < w_{e^*}$ for all $x \in \mathcal{E}$. \label{halfSpace}
\end{prop}
\begin{proof}
Let $G$ be the set of all edges with weights greater than the bottleneck edge $e^*$, and $L$ the set of all edges with weights less than the bottleneck edge $e^*$. By the definition of the bottleneck edge, there is guaranteed to be a maximum matching with all edges in the set $\{L \cup e^*\}$, and no maximum matching with all edges in the set $L$. By the definition of the set $G$, the maximum matching of edges from the set $\{L \cup e^*\}$ will all have lower weights than any edges in the set $G$. Thus, every edge in the set $G$ is excluded from being the bottleneck. Therefore, the ordering of the edges within the sets $G$ and $L$ has no impact on the bottleneck.
\end{proof}
\begin{remark*}
Proposition~\ref{halfSpace} demonstrates that the allowable perturbations on all edges with weight greater than the bottleneck are unbounded above, and the allowable perturbations on all edges with weights less than the bottleneck are unbounded below. In light of this, we will refer to the allowable perturbation sets as \emph{half space constraints}, because they will only specify an upper or lower bound with the other bound being negative or positive infinity.
\end{remark*}

Finally, note that the method for constructing the set $\Lambda _\Delta$ and the subsequent expansion of the allowable perturbation intervals has nearly the order computational complexity as the BAP, because it involves solving the bottleneck assignment problem on two subgraphs. One of these subgraphs, $\mathcal{G}^+$, is only one edge smaller than the original. However, by using only Proposition~\ref{halfSpace}, we can construct a different set of allowable perturbation intervals, as shown in Algorithm~\ref{alg:fastBound}.
\begin{algorithm}
 \KwData{($e^*,\mathcal{W})$}
 $\mathcal{U}$  $\backslash \backslash$ \emph{Upper bound on the allowable perturbations} \;
 $\mathcal{L}$  $\backslash \backslash$ \emph{Lower bound on the allowable perturbations} \;
 $\Delta^+ \leftarrow \frac{1}{2}\min(|w_j-w_{e^*}|),  \forall \{j \mid w_j \geq e^*\})$ \;
 $\Delta^- \leftarrow \frac{1}{2}\min(|w_j-w_{e^*}|),  \forall \{j \mid w_j \leq e^*\})$ \;
\For{$w_j\in\mathcal{W}$}{

	\uIf{$w_j > w_{e^*}$}{
		$\mathcal{L}[j] = (w_{e^*} + \Delta^+) - w_j$ \;
	}
	\uElseIf{$w_j < w_{e^*}$}{
		$\mathcal{U}[j] = (w_{e^*} - \Delta^-) - w_j$ \;
	}
	\Else{
		$\mathcal{L}[j] = -\Delta^-$ \;
		$\mathcal{U}[j] = \Delta^+$ \;
	}
}
return $\mathcal{U},\mathcal{L}$ \;

\caption{Naive Algorithm for Robustness \label{alg:fastBound}}
\end{algorithm}

\begin{prop}
Algorithm~\ref{alg:fastBound} constructs a set of allowable perturbation intervals in $\mathcal{O}(N)$ where $N=|\mathcal{E}|$. \label{prop:HalfSpace}
\end{prop}

\begin{proof}
Define $w_x$ to be the weight of an edge $x\in\mathcal{E}$, $\bar{w}_x$ to be the perturbed weight of that edge, and $e^* \in \mathcal{E}$ be the bottleneck edge as defined in \eqref{eq:i}. Define the sets of upper and lower bounds returned by Algorithm~\ref{alg:fastBound} as $\mathcal{U}$ and $\mathcal{L}$. Let $(\mathcal{L}_j,\mathcal{U}_j)$ represent the allowable perturbation for edge $j$, such that $\bar{w}_j - w_j \in (\mathcal{L}_j,\mathcal{U}_j)$.\\

Let $G$ be the set of edges with weights greater than the bottleneck, and $L$ the set of edges with weight less than the bottleneck. According to Proposition~\ref{halfSpace}, if these sets are unchanged by the perturbation then the bottleneck assignment is robust to that perturbation. Let $w_j > w_{e^*}$ be the weight of an edge in set $G$, 
\begin{align*}
\bar{w}_j > w_j+\mathcal{L}_j = w_j + (w_{e^*} + \Delta^+ - w_j) > \bar{w}_{e^*}. 
\end{align*}
Therefore, $\bar{w}_j > \bar{w}_{e^*}$ for all elements of $G$, so $G$ is unchanged by the perturbation. Let $w_j < w_{e^*}$ be the weight of an edge in set $L$, 
\begin{align*}
\bar{w}_j < w_j+\mathcal{U}_j = w_j + (w_{e^*} - \Delta^- - w_j) < \bar{w}_{e^*}. 
\end{align*}
Therefore, $\bar{w}_j < \bar{w}_{e^*}$ for all elements of $L$, so $L$ is unchanged by the perturbation.$G$ and $L$ are unchanged by the perturbation, so the bottleneck assignment is robust to this set of perturbations.\\

Algorithm~\ref{alg:fastBound} constructs these bounds by performing one pass through the weight set to find the minimum distance weight to the bottleneck, and another pass through the edges to construct their upper and lower bounds. Therefore, it has complexity $\mathcal{O}(|\mathcal{E}|)$.
\end{proof}

\section{Numerical Example}

First, to illustrate the various bounds, extensions, and applications, we will analyze a small example. Consider a fully connected multi-agent system with 3 agents and 4 tasks with cost matrix 
\begin{align*} 
C = 
\begin{bmatrix}
64.5 && 79.2 && 25.0 && 9.8 \\
85.9 && 81.2 && 21.5 && 28.3 \\
47.1 && 12.1 && 41.3 && 35.7
\end{bmatrix}.
\end{align*}

The first step towards defining the set of allowable perturbations is to run the bottleneck assignment algorithm on the full graph to find $e^*$. The bottleneck edge $e^*$ this example is $(2,3)$ with weight $21.5$. The cost matrices of the two subgraphs $\mathcal{G^-}$ and $\mathcal{G^+}$ are 
\begin{align*}
C^- &= 
\begin{bmatrix}
64.5 && 79.2 && 9.8 \\
47.1 && 12.1 && 35.7
\end{bmatrix}, \\
C^+ &= 
\begin{bmatrix}
64.5 && 79.2 && 25.0 && 9.8 \\
85.9 && 81.2 && \infty && 28.3 \\
47.1 && 12.1 && 41.3 && 35.7
\end{bmatrix}.
\end{align*}

Running the bottleneck algorithm on the two subgraphs gives $e^- = (3,2)$ with weight $w_{e^-} = 12.1$ and $e^+ = (2,4)$ with weight $w_{e^+} = 28.3$. To find the set of allowable perturbation intervals we first use Algorithm~\ref{alg:slowBound}, because we have already identified $e^-$, $e^*$, and $e^+$. The results are presented in Table~\ref{tble:RelaxedDeltas}.
\begin{table}[ht]
\renewcommand{\arraystretch}{1.5}
\begin{center}
\caption{Allowable Perturbations from Algorithm~\ref{alg:slowBound} \label{tble:RelaxedDeltas}}
\resizebox{0.5\textwidth}{!}{
\begin{tabular}{l c  c  c  c }
 & T$_1$ & T$_2$ & T$_3$ & T$_4$ \\
 A$_1$ & $[-39.6,\infty]$ & $[-54.3,\infty]$ & $[-\infty,\infty]$ & $[-\infty,7.0]$ \\ 
 A$_2$ & $[-61.0,\infty]$ & $[-56.3,\infty]$ & $[-4.7,3.4]$ & $[-3.4,\infty]$ \\ 
 A$_3$ & $[-22.2,\infty]$ & $[-\infty,4.7]$ & $[-16.4,\infty]$ & $[-10.8,\infty]$ \\ 
\end{tabular}
}
\end{center}
\end{table}

Next, we compare the allowable perturbation intervals from running Algorithm~\ref{alg:fastBound} over the same graph. The results are presented in Table~\ref{tble:fastBound}
\begin{table}[ht]
\renewcommand{\arraystretch}{1.5}
\begin{center}
\caption{Allowable Perturbations from Algorithm~\ref{alg:fastBound} \label{tble:fastBound}}
\resizebox{0.5\textwidth}{!}{
\begin{tabular}{l c  c  c  c }
 & T$_1$ & T$_2$ & T$_3$ & T$_4$ \\
 A$_1$ & $[-41.25,\infty]$ & $[-55.95,\infty]$ & $[-1.75,\infty]$ & $[-\infty,7.0]$ \\ 
 A$_2$ & $[-62.65,\infty]$ & $[-57.95,\infty]$ & $[-4.7,1.75]$ & $[-5.05,\infty]$ \\ 
 A$_3$ & $[-23.85,\infty]$ & $[-\infty,4.7]$ & $[-18.05,\infty]$ & $[-12.45,\infty]$ \\
\end{tabular}
}

\end{center}
\end{table}

Note that for some edges Algorithm~\ref{alg:fastBound} and Algorithm~\ref{alg:slowBound} provide the same bounds. However, the smallest intervals are larger using Algorithm~\ref{alg:slowBound}. Some edges are also entirely unbounded using Algorithm~\ref{alg:slowBound}.\\

In both algorithms edges above the bottleneck can be infinitely positively perturbed, and edges below the bottleneck can be likewise infinitely negatively perturbed. Therefore, to compare the algorithms, we propose examining the magnitude of allowable perturbation in the non-infinite direction. However, note that this method does not capture the number of unbounded edges from Algorithm~\ref{alg:slowBound}.\\

Fig.~\ref{fig:minRobustness} presents simulation results. The graphs generated have $|\mathcal{V_A}|=|\mathcal{V_T}|=[3,100]$ nodes and are fully connected $\mathcal{E}=\mathcal{V}_A\times \mathcal{V}_T$. The plots show the minimum magnitude allowable perturbation, or the ``tightest" allowable perturbation bound. For each graph size from $[3,100]$ we randomly generated 1000 graphs with weights randomly sampled from $[0,100]$.

\begin{figure}[thpb]
  \centering
  \includegraphics[scale=0.55]{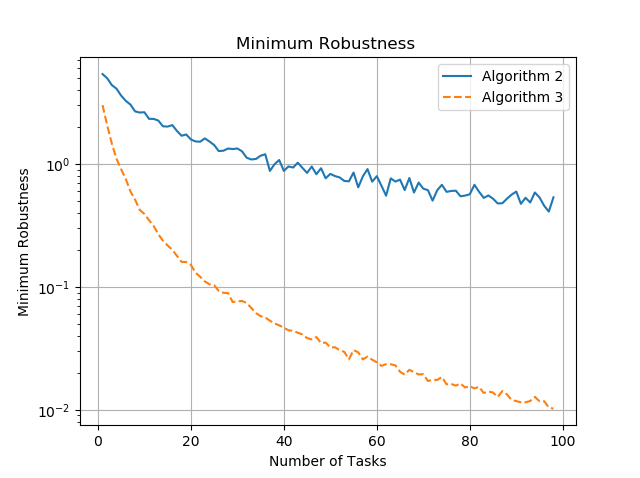}
  \caption{Minimum size of allowable perturbations}
  \label{fig:minRobustness}
\end{figure}

Algorithm~\ref{alg:slowBound} is provides a larger minimum size intervals as the problem size grows. The minimum robustness of Algorithm~\ref{alg:fastBound} is driven entirely by the density of the weights, and the larger the problems the more likely there are many weights near the bottleneck. However, Algorithm~\ref{alg:slowBound}'s minimums also drop as the problem size grows larger, because its performance is tied to the density of the greatest weight edges in the maximum matchings, which will also increase with the number of tasks in the graph.\\

\section{Concluding Remarks}

In this paper the robustness of the bottleneck assignment problem (BAP) to a set of perturbations is analyzed. Two methods of approximating the set of allowable perturbations are proposed, one with the same order computational complexity as the BAP, and the other is linear in the number of edges in the graph. These approaches differ from the previous approaches in literature by allowing all weights to vary simultaneously and independently. Simulations showed that the higher order complexity bound provided better worst case robustness. The robustness of the graph can be used to guarantee the optimality of the bottleneck assignment if the errors in each weight fall within the set of allowable perturbations. \\

In future work, similar robustness for the linear assignment problem, the general assignment problem, and various other assignment problem formulations will be investigated. Another possible future problem would be examining the geometric implications of these robustness intervals, perhaps informing a method of deploying multi-agent system in a pattern likely to lead to robust solutions.




\bibliographystyle{IEEEtran}
\bibliography{IEEEabrv,robustAssignment}

\end{document}